\documentclass{article}
\usepackage{amsmath}
\usepackage[margin=0.5in]{geometry}
\usepackage{verbatim}
\usepackage{amssymb}
\usepackage{pdfpages}
\usepackage{comment}
\usepackage{amsthm}
\usepackage{pgfplots}
\usepackage{pgfplotstable} 
\usepackage{graphicx}
\usepackage[english]{babel}
\usepackage{subcaption}
\usepackage{float}
\usepackage{tkz-euclide}
\usepackage{tikz}
\usetikzlibrary{patterns}
\usetikzlibrary{decorations,backgrounds}
\usepackage{color}

\newcommand{\RM}{\mathrm{RM}}
\newcommand{\T}{\mathcal{T}}
\newcommand{\E}{\mathcal{E}}

\newcommand{\RE}{\mathbb{R}}

\newcommand{\vertiii}[1]{{\left\vert\kern-0.25ex\left\vert\kern-0.25ex\left\vert #1 
    \right\vert\kern-0.25ex\right\vert\kern-0.25ex\right\vert}}
\newcommand{\bfu}{\boldsymbol{u}}
\newcommand{\bff}{\bar{f}}
\newcommand{\bfv}{\boldsymbol{v}}
\newcommand{\bfz}{\boldsymbol{z}}
\newcommand{\bfepsilon}{\boldsymbol{\varepsilon}}

\newcommand{\bfw}{\boldsymbol{w}}
\newcommand{\Uh}{\boldsymbol{U}_h^{k,j}} %K

\let\div\undefined
\DeclareMathOperator{\div}{div}

\newtheorem{theorem}{Theorem}
\newtheorem{lemma}[theorem]{Lemma}

\theoremstyle{remark}
\newtheorem{remark}{Remark}
\theoremstyle{definition}
\newtheorem{definition}{Definition}
\newcommand{\bu}{{\boldsymbol u}}
\newcommand{\bv}{{\boldsymbol v}}
\newcommand{\bw}{{\boldsymbol w}}
\newcommand{\bU}{{\boldsymbol U}}
\newcommand{\bV}{{\boldsymbol V}}

\numberwithin{equation}{section}
\numberwithin{theorem}{section}
\numberwithin{remark}{section}
\numberwithin{definition}{section}

\begin{document}

\title{Superconvergence of DPG approximations in linear elasticity}
\author{Fleurianne Bertrand, Henrik Schneider}
\date{}
\maketitle

{\bf{Abstract:}} 
Existing a priori convergence results of the discontinuous Petrov-Galerkin method to solve the problem of linear elasticity are improved. Using duality arguments, we show that higher convergence rates for the displacement can be obtained. Post-processing techniques are introduced in order to prove superconvergence and numerical experiments {\color{black} confirm} our theory.

%
%\begin{resume} 
%Les résultats de convergence s'applicant aux méthodes de Petrov-Galerkine discontinues sont ameliorées pour  les équations de l'élasticité linéaire. En utilisant un argument de dualité, nous prouvons que la convergence sur les déplacements est d'ordre plus élevé. Le  post-traitement de la solution permet de prouver des résultats de super-convergence et les examples numériques valident notre théorie.
%\end{resume}
%
%\subjclass{65N30}
%
%\keywords{discontinuous Petrov-Galerkin, linear elasticity}
%

%
% There should be no $\Gamma_D,1$ anymore.

\section*{Introduction}

% DPG intro
Finite element approximation of partial differential equations using the DPG
method is a popular and effective technique introduced in \cite{DPG1,DPG2,DPG3}.
It is a minimal residual method with broken test spaces and can therefore be
seen simultaneously as a Least-Squares method and as a saddle-point method. 
{\color{black} The critical idea is the optimal test function approximation that guarantees
discrete stability. This 
major advantage allow the use of
broken test spaces consisting of functions with no continuity requirements at
element interfaces.} An elementary characterisation of the natural norms
on these interface spaces is provided in \cite{Maxwell}. Moreover, sufficient
conditions under which stability of broken forms {\color{black}follows} from {\color{black}the} stability of their
unbroken relatives were stated.

% Solid mechanics
 The inherent stability properties of the DPG method make it a promising approach in solid mechanics, especially for the determination of vibrations of elastic structure due to the possibility of obtaining a pointwise symmetric
approximation for stresses {\color{black}
in a stable way}. Moreover, the simultaneous 
approximation of the stress-tensor $\sigma$ and the displacement $u$, the DPG method also yields robust error bounds and the contributions
\cite{BDGQ12,Carstensen2016} confirm their suitability in
computational mechanics. However, although the stress-tensor and the displacement are related by the strain-stress relationship $\sigma = \mathbb C \varepsilon (u)$
with the symmetric gradient  $\varepsilon(u) = \frac 1 2 (\nabla u + (\nabla u)^T) $, both functions are usually approximated simultaneously with the same order. The
first aim of this paper is therefore to prove better convergence rates for the displacement variable $u$, either by increasing the polynomial order of the corresponding
approximation space or by defining an approximation of the scalar field variable by suitable postprocessing. This issue has been raised and addressed for the Poisson
problem in \cite{F18} {\color{black}and for the general second order case in \cite{ultraweak-duality}. A similar result for superconvergence for the primal DPG formulation was devoloped in \cite{BGH14}.} It is also closely related to the refined error estimates of the Least-Squares method of \cite{BHsuperclosedness}.

% Longterm vision
As a second motivation, the refined error estimates proved in this paper also allow us to pave the way for the consideration of the DPG method for the determination of vibrations
of elastic structures. Dual-mixed formulations have been considered in
the solid for the elastoacoustic source problem (see,e.g.,
\cite{Gatica2007,Gatica2012}) but their extension to {\color{black}the} eigenvalue problem does not
fit the existing theories for mixed eigenvalue problems. DPG-based formulations would fit this theory, because of the gained
flexibility of the finite element spaces. However, similarly to the issues arising in the Least-Squares context (see \cite{imals, BLS4eigelas, BLS4eigelas2}), superconvergence results are crucial to prove the
convergence of eigenvalue approximation with the DPG method {\color{black}(see \cite{BBS22})}, and the aim of this paper.

\section{The DPG formulation}
The DPG method under consideration is based on the first-order system of the linear elasticity equations in a domain $\Omega$
\begin{subequations}
    \label{eq:source-problem}
    \begin{align}
        \mathcal A \sigma - \varepsilon(u)   &= %\mathcal{A}\bff
        0 \quad \text{in}\  \Omega,\\ 
        \label{eq:momentum}
        -\div \sigma &= f \quad \text{in} \  \Omega ,\\
        u &= 0 \quad \text{on}\  \Gamma_D,\\
        \sigma \cdot \nu &= 0 \quad \text{on}\  \Gamma_N,
    \end{align}
\end{subequations}
where $\mathcal A :
\mathbb{S} \rightarrow \mathbb{S} $  is the compliance tensor given by $\tau \mapsto \frac{1}{2\mu} \tau  -
\frac{\lambda}{2\mu(2\mu+d\lambda)} \mathrm{tr}(\tau) I_{d\times d}$, with the Lamé parameter $\lambda$ and $\mu$. 
The extension for asymmetric matrices is given by $\mathcal A : \mathbb M \rightarrow \mathbb M$ $\tau \mapsto
\mathcal{A}(\mathrm{sym}(\tau)) + \mathrm{as}(\tau)$.
We assume that the domain $\Omega \subset  \RE^d {\color{black} =:\mathbb V} $ is a Lipschitz domain and that the space dimension $d$ equals 2 or 3. The boundary $\partial \Omega$ consists of two open subsets $\Gamma_{D}$ and $\Gamma_{N}$, where Dirichlet and Neumann conditions are prescribed.
The boundary part $\Gamma_{D}$, where the elastic body is clamped, is assumed to be of positive measure and $\Gamma_{N}$ is assumed to be its complement $\partial \Omega \backslash \Gamma_{D}$, i.e. $\partial \Omega = \overline{\Gamma_D \cup \Gamma_N}$, $\Gamma_D \cap
\Gamma_N = \emptyset$. Under these assumptions, problem \eqref{eq:source-problem} admits a unique solution $(\sigma, u) \in
 H_{\Gamma_N}(\div,\mathbb S)
\times
H^1_{\Gamma_D}(\Omega,{\color{black}\mathbb V})$
 for all $f\in L^2(\Omega;{\color{black}\mathbb V})$
 with 
\begin{align*}
    &{\color{black}L^2(\Omega;\mathbb V) = \left[L^2(\Omega)\right]^d},\\
    &{\color{black}H^1_{\Gamma_D}(\Omega;\mathbb V) =  \left[H^1_{\Gamma_D}(\Omega)\right]^d,}\\
    &H_{\Gamma_N}(\operatorname{div}, \Omega ; \mathbb{S})=\left\{\sigma \in L^{2}(\Omega ; \mathbb{S}): \operatorname{div} \sigma \in L^{2}\left(\Omega,  {\color{black}\mathbb V} \right), \ \sigma \cdot \nu = 0 \text{ on }\Gamma_N\right\} .
\end{align*}
Throughout this paper, the domain $\Omega$ is assumed to be such that the following regularity assumption holds
\begin{align}
\label{eq:regularity}\Vert u \Vert_{H^2(\Omega)} + \Vert \sigma\Vert_{H^1(\T)} &\leq C \Vert f\Vert  .
\end{align}
{\color{black}
\begin{remark}
This regularity estimate is fulfilled for $d=2$ if $\Omega$ is a convex polyhedral domain. {\color{black} If $u\in H^1_{\Gamma_D}(\Omega;\mathbb V)$ solves
    \begin{align*}
        -\div (\mathbb C \varepsilon(u)) = f \in L^2(\Omega;\mathbb V),
    \end{align*}
    then $u\in H^2(\Omega;\mathbb V)$ and $\Vert u \Vert_{H^2(\Omega;\mathbb V)} \leq C_1 \Vert f\Vert$, see \cite{G85}.} Using the stress-strain relationship yields 
    \begin{align*}
        \Vert\sigma\Vert_{H^1(\T)}\ = \Vert\mathcal A\varepsilon(u)\Vert_{H^1(\T)} \leq C_2 \Vert u \Vert_{H^2(\Omega)} \leq C_1C_2 \Vert f\Vert,
    \end{align*}
    where $C_2$ is robust for $\lambda \rightarrow \infty$.
\end{remark}
}
%The above mentioned stress properties imply that the exact stress $\sigma$ on an elastic body occupying $\Omega \subseteq \mathbb{R}^{N}$ lies in the space

% Ultraweak

The ultra-weak formulation under consideration is  derived from \eqref{eq:source-problem} by testing  with broken test functions, based on a shape-regular simplicial triangulation $\mathcal T$. It allows for the broken test spaces
\begin{subequations}
    \label{eq:brokentestspaces}
    \begin{align}
    H^{1}(\mathcal{T}) &:=\left\{v \in L^{2}(\Omega):\left.v\right|_{T} \in H^{1}(T)\ {\color{black}\forall }T \in \mathcal{T}\right\}, \\ {H}(\operatorname{div} ; \mathcal{T}) &:=\left\{{\tau} \in {L}^{2}(\Omega):\left.{\tau}\right|_{T} \in {H}(\operatorname{div} ; T)\ {\color{black}\forall} T \in \mathcal{T}\right\}     
    \end{align}
\end{subequations}
and for the piecewise differential operators $\nabla_{\mathcal{T}}: H^{1}(\mathcal{T}) \rightarrow {L}^{2}(\Omega)$ and  $\operatorname{div}_{\mathcal{T}}: {H}(\operatorname{div} ; \mathcal{T}) \rightarrow L^{2}(\Omega)$
defined on each element $T\in \mathcal T$ by
\begin{align}
\nabla_{\mathcal{T} } v|_{T}:=\nabla\left(\left.v\right|_{T}\right),\left.\quad \operatorname{div}_{\mathcal{T}} {\tau}\right|_{T}:=\operatorname{div}\left(\left.{\tau}\right|_{T}\right) .
\end{align}

Using these broken spaces as test spaces has the advantage that the field variables $(\sigma, u)$ can be sought in $L^2(\Omega)$ allowing for discontinuous {\color{black} trial} functions. A weak continuity condition is imposed with the introduction of trace variables living in the following skeleton spaces: 
\begin{subequations}
\begin{align}
H^{1 / 2}_{\Gamma_D}(\partial\mathcal T) &:=\left\{\widehat{u} \in \Pi_{T \in \mathcal{T}} H^{1 / 2}(\partial T): \exists w \in H^{1}_{\Gamma_D}(\Omega) \text { such that }\left.\widehat{u}\right|_{\partial T}=\left.w\right|_{\partial T} \forall T \in \mathcal{T}\right\}, \\
H^{-1 / 2}_{\Gamma_N}(\partial\mathcal T) &:=\left\{\widehat{\sigma} \in \Pi_{T \in \mathcal{T}} H^{-1 / 2}(\partial T): \exists \boldsymbol{q} \in \boldsymbol{H}_{\Gamma_N}(\operatorname{div} ; \Omega) \text { such that }\left.\widehat{\sigma}\right|_{\partial T}=\left.\left(\boldsymbol{q} \cdot \boldsymbol{n}_{T}\right)\right|_{\partial T} \forall T \in \mathcal{T}\right\}\ .
    \end{align}
\end{subequations}
{\color{black}
With the Sobolev spaces
\begin{subequations}
\begin{align}
    &L^2(\Omega;\mathbb M)= \left[L^2(\Omega)\right]^{d\times d},\\
    &
    L^2(\Omega;\mathbb A)= \{v \in L^2(\Omega, \mathbb M {\color{black})} \ : \ v = -v^T\},\\
    &H(\div,\T;\mathbb S) = \{v \in \left[H(\div,\T)\right]^d \ : \ v = v^T\}\ ,
\end{align}
\end{subequations}
}
our DPG formulation seeks $(\sigma, u) \in
L^2(\Omega; \mathbb M) \times L^2(\Omega;\mathbb V) $ as well as $(\hat\sigma_n, \hat u) \in {\color{black} H^{-1/2}_{\Gamma_N}(\partial \T) \times H^{1/2}_{\Gamma_D}(\partial\T)}$
such that 
\begin{align}\label{eq:dpg}(\mathcal A\sigma,\tau)  
    + (u,\div_\T \tau) + (\sigma, \nabla_\T v) +(\sigma,q) 
    - \langle \hat u, \tau\cdot\nu\rangle_{\partial\T} - \langle\hat\sigma_n,v\rangle_{\partial\T}
     = (f,v)
\end{align}
    holds for all $(\tau,v,q) \in  H(\div, \T;\mathbb S) \times H^1(\T;\mathbb V)\times L^2({\color{black}\Omega} ; \mathbb A)$. 
To simplify the notation, we also introduce the spaces ${\bU} = L^2(\Omega; \mathbb M) \times L^2(\Omega;\mathbb V)   \times
      {\color{black}H^{-1/2}_{\Gamma_N}(\partial \T) \times  H^{1/2}_{\Gamma_D}(\partial\T) }
$ and $
{\bV} = H(\div, \T;\mathbb S) \times H^1(\T;\mathbb V)\times L^2({\color{black}\Omega ;}\mathbb A)$ 
as well as the bilinear forms
\begin{subequations}
\begin{align}
    b({\bfu},{\bfv}) &= (\mathcal A\sigma,\tau)  
    + (u,\div_\T \tau) + (\sigma, \nabla_\T v) +(\sigma,q) 
    - \langle \hat u, \tau\cdot\nu\rangle_{\partial\T} - \langle\hat\sigma_n,v\rangle_{\partial\T},\\
    l({\bfv} ) &= (f,v)\ .
    \end{align}
\end{subequations}
for ${\bfu} = (\sigma,u,\hat\sigma_n, \hat u)$ and ${\bfv} = (\tau,v,q)$. With these notations, the variational formulation \eqref{eq:dpg} allows for the abstract form
\begin{align}
    \label{eq:dpg-abstract}
    b({\bfu},{\bfv}) = l({\bfv})\quad \forall {\bfv} \in {\color{black}  \bV} .
\end{align}

The well-posedness of this formulation is shown in \cite{BDGQ12} and the proof relies on the following three key properties: 
\begin{enumerate}
    \item uniqueness 
\begin{align}
    \{ \bw \in \bU \ : \ b(\bw,\bv) = 0, \ \forall \bv\in \bV\} = \{ 0\}, \label{ass:I}
\end{align}
\item inf-sup condition
\begin{align}
    \Vert \bv \Vert_\bV \lesssim \sup_{\bw\in \bU \setminus \{0\}} \frac{b(\bw,\bv)}{\Vert \bw\Vert_\bU} \quad \forall \bv \in \bV,\label{ass:II}
\end{align}
\item continuity of $b$
\begin{align}
    b(\bw,\bv) \lesssim \Vert \bw\Vert_\bU\Vert \bv\Vert_\bV \quad \forall \bw\in \bU,\ \forall \bv\in \bV \label{ass:III},
\end{align}
\end{enumerate}
with the norms
\begin{subequations}
\begin{align}
    \Vert (\sigma, u, \hat\sigma_n, \hat u)\Vert_\bU^2 &:= \Vert \sigma\Vert^2 
                                                        + \Vert u \Vert^2
                                                        + \Vert \hat\sigma_n\Vert^2_{H^{-1/2}(\partial\T)}
                                                        + \Vert \hat u \Vert^2_{H^{1/2}(\partial\T)},\\
    \Vert (\tau,v,q)\Vert_\bV^2 &:= \Vert \nabla_{\T} v\Vert^2 +
                                   \Vert v \Vert^2 +
                                   \Vert \div_{\T} \tau\Vert^2 +
                                   \Vert \tau\Vert^2 +
                                   \Vert q \Vert^2.
\end{align}\end{subequations}
In particular, the operator $B : \bU \rightarrow \bV^\prime$ defined by $(B \bu)(\bv) = b(\bu, \bv)$ and the trial-to-test operator $\Theta:\bU\rightarrow \bV$ defined by
\begin{align}
\label{eq:theta}
    (\Theta \bfw, \bfv)_\bV = b(\bfw,\bfv) \qquad \forall \bfv \in \bV
\end{align}
allow for a constant $C>0$ such that
\begin{align}
C^{-1}\|\boldsymbol{u}\|_{\bU}^{2} \leq\|B \boldsymbol{u}\|_{\bV^{\prime}}^{2}=b(\boldsymbol{u}, \Theta \boldsymbol{u}) \leq C\|\boldsymbol{u}\|_{\bU}^{2} 
\end{align}
holds for all $\bu\in\bU$.

\begin{comment}
We introduce now the broken ultra-weak DPG formulation
for this problem. Let 
\begin{align*}
    U_0 &= L^2(\Omega; \mathbb M) \times L^2(\Omega;\mathbb V)   \\
    \hat U &=  H^{1/2}_{\Gamma_D}(\partial\T) \times H^{-1/2}_{\Gamma_N}(\partial \T), \\
    U &= U_0 \times \hat U \\
    V &= H(\div, \T;\mathbb S) \times H^1(\T;\mathbb V)\times L^2({\color{black}\T ;}\mathbb A)
\end{align*}
Then we can formulate the ultra-weak formulation as follows
\begin{align*}
    b_0((\sigma,u),(\tau,v,q)) &:= 
    (\mathcal A\sigma,\tau)  
    + (u,\div \tau) + (\sigma, \nabla_\T v) +(\sigma,q)  \\
    \hat b ((\hat\sigma_n, \hat u),(\tau,v)) &:= - \langle \hat u, \tau\cdot\nu\rangle_{\partial\T} - \langle\hat\sigma_n,v\rangle_{\partial\T}\\
    a((x,y,z),(\hat x,\hat y, \hat z)) &:= (z,\hat z)_V - \overline{b_0(\hat x,z)} - \overline{\hat b(\hat y, z)}
                                            + b_0(x,\hat z) + \hat b(y, \hat z) \\ 
                                            &\qquad\text{where} (x,y,z),(\hat x,\hat y, \hat z)\in U_0 \times \hat U\times V \\
    l(v) &:= (f,v) + (\bff,\mathcal{A}\tau) .
\end{align*}
\end{comment}

\section{Finite element approximation}
{\color{black} 
Let $\bU_h\subset \bU$ be a  conforming finite-dimensional subspace of $\bU$. 
The ideal DPG method now seeks $\bfu_h \in \bU_h$ such that
\begin{align}\label{eq:dpghI}
    b(\bfu_h,\Theta \bfw_h) = l(\Theta \bfw_h)\quad {\color{black}\forall } \bfw_h\in \bU_h. 
\end{align}
To obtain the practical DPG method, the test space is replaced by a finite element space $\bV_h
\subset \bV$ with
}
 the following crucial compatibility condition on the spaces: there exists a Fortin {\color{black}operator} $\Pi : \bV \rightarrow \bV_h$ such that there exists a constant $C_\Pi>0$ with
\begin{align}
    b(\bu_h,\bv-\Pi\bv)=0 \text{ and }
    \|\Pi\bv\|_\bV \leq C_\Pi \|\bv\| \quad {\color{black} \forall} \bu_h \in \bU_h, \bv\in \bV .
\end{align}
The discrete trial-to-test operator $\Theta_{h}:\bU_h \rightarrow V_{h}$ is defined through
$$
\left(\Theta_{h} \bfw_{h}, \bfv_{h}\right)_{\bV}=b\left(\bfw_{h} , \bfv_{h} \right) \text { for all } \bfv_{h} \in \bV_{h} 
$$
and the practical DPG formulation seeks $\bfu_h \in \bU_h$ such that {\color{black}\
 \begin{align}\label{eq:dpgh}
    b(\bfu_h,\Theta_h \bfw_h) = l(\Theta_h \bfw_h)\quad \forall  \bfw_h\in \bU_h. 
\end{align}
}
With respect to the simplicial triangulation $\mathcal T$, our focus is on approximations within the spaces of piecewise polynomials of degree at most  $k \in \mathbb{N}_{0}$, given by
\begin{align}
P^{k}(T)&=\left\{v \in L^{\infty}(T) \mid v \text { is polynomial on } T \text { of degree } \leq k\right\}, \\
P^{k}(\mathcal{T}) &=\left\{v_{\mathcal T} \in L^{\infty}(\Omega)\left|\forall T \in \mathcal{T}, v_{\mathcal T}\right|_{T} \in P^{k}(T)\right\},\\
P^{k}(\T;\mathbb M) &= (P^k(\T))^{d\times d}, \quad
P^k(\T;\mathbb V) = (P^k(\T))^d,  \\
&\!\!\!\!\!\!\!\!\!\!\!\!\!\!\!\!\!\!\!\!\!\!\! {\color{black} P^k(\T; \mathbb S) = \{ v \in P^k(\T;\mathbb M)\ : \ v = v^T\},  \quad P^k(\T; \mathbb A) =\{v\in P^k(\T;\mathbb M) \ :\ v = - v^T\} } \ .
\end{align}
For the approximation of the trace variables, $\mathcal E$ denotes the set of all sides in the triangulation and we define the approximation spaces
\begin{align}
P^{k}(\partial T) &= \{v \in L^\infty(\partial T): v_{|F} \in P^k(F) \ \text{for all (d-1)-dimensional subsimplices $F$ of $T$}  \}, \\
{\color{black} 
S^{k}_{\Gamma_D}(\E;\mathbb V)} &= 
\{
v\in L^\infty(\partial \T;\mathbb V) 
: v|_{\partial T} \in (P_k(\partial\T))^d \cap C^0(\partial T) \
\forall T \in \T
\}  \
\cap \ H^1_{\Gamma_D}(\Omega)
,   \\  
P^k_{\Gamma_N}(\E;\mathbb V)  &= \{v\in L^\infty(\partial \T;\mathbb V) : v|_{\partial T} \in (P_k(\partial\T))^d\ \forall T \in \T, \  \text{for all edges}\ E \subset \Gamma_N: \ v(E) =0 \}.
\end{align}
The DPG formulation \eqref{eq:dpgh} allows for a natural choice of piecewise polynomial trial spaces with
\begin{subequations}
\begin{align}
    \bU_{h}^{k,j}  &:= P^{k}(\T;\mathbb M) \times P^{k+j}(\T;\mathbb V) \times
       {\color{black}P^k_{\Gamma_N}(\E;\mathbb V)\times S^{k+1}_{\Gamma_D}(\E;\mathbb V)}, \\
    \bV_h^k &:= P^{k+2} (\T;\mathbb S) \times P^{k+d} (\T;\mathbb V) \times P^k(\T;\mathbb A) 
\end{align}
\end{subequations}
where $k \geq 0$ and $j=0,1$. Note that the original formulation in \cite{GQ14} concerns only $j=0$, 
i.e. approximations of $u$ and the stress tensor {\color{black}$\sigma$} in polynomial spaces of the same order. 
However, recalling the constitutive law $\mathcal A \sigma = \varepsilon(u)$, this can appear 
suboptimal and we will show that a better convergence rate for $j=1$ can be obtained in {\color{black}Section} 
\ref{sec:superconvergence}. {\color{black} Additionally $U_h^{k,j}$ denotes the second component of $\bU_{h}^{k,j}$.}
The next lemma recalls the definition of the Fortin operator $\Pi$. 
\begin{lemma}[Fortin operator]
Let $k\in \mathbb N_0$, $j=0,1$, $\Pi_k : L^2(\Omega) \rightarrow P^k(\T)$ denote the
$L^2$ projection and $\Pi_k^{\div,\mathbb S} : H(\div, \T;\mathbb S) \rightarrow P^k(\T;\mathbb S)$ denotes
the symmetric divergence projection \cite[Lemma 4.1, Theorem 3.6]{GQ14,GZ11} such that
for all $\tau \in H(\div,\T;\mathbb S)$:
\begin{subequations}\label{eq:div-sym}
\begin{align}
    \Vert \Pi_{k+2}^{\div,\mathbb S} \tau \Vert_{H(\div,\T)} &\leq C \Vert\tau\Vert_{H(\div, \T)}, \\
    (\tau-\Pi_{k+2}^{\div,\mathbb S}  \tau,\tau_h) &= 0 \quad\ \forall \tau_h \in P^k(\T;\mathbb S) \label{eq:approx_Pi_div},\\
     \langle(\tau -\Pi_{k+2}^{\div,\mathbb S}  \tau ) \cdot n, \mu_h\rangle_{\partial \T} &= 0,
    \quad\forall \mu_h \in P^{k+1}(\partial \T),\\
    \text{ and }
 \div_\T \Pi_{k+1}^{\div,\mathbb S} \tau &= \Pi_k \div_\T \tau. \label{eq:commu_Pi_div}
\end{align}
\end{subequations}
Moreover, let $\Pi^{\mathrm{grad}}_{k+d}: H^1(\T) \rightarrow P^{k+d}(\T)$ {\color{black}be} such that
\begin{align*}
    \Vert \Pi_{k+d}^{\mathrm{grad}} \Vert_{H^1(\T)} &\leq C \Vert u \Vert_{H^{1}(\T)},\\
    (u-\Pi^{\mathrm{grad}}_{k+d} u,u_h) &= 0 \quad \forall u_h \in P^{k-1}(\T),\\
    \text{ and }
    \langle u-\Pi^{\mathrm{grad}}_{k+d} u,\mu_h) &= 0 \quad \forall \mu_h \in P^{k}(\partial\T)
\end{align*} 
for all {$u \in H^1(\T)$}.
The operator $\Pi (\tau,v,q) := (\Pi^{\div,\mathbb S}_{k+2} \tau,
\Pi^{\mathrm{grad}}_{k+d} v, \Pi^k q)$ is a Fortin operator, i.e. 
\begin{align*}
 b(\bu_h,\bv-\Pi\bv)=0 
 \end{align*} 
 and $\|\Pi\bv\|_\bV \leq C_\Pi \|\bv\| \quad \text{for all } \bu_h \in \bU_h^{k,j}, \bv\in \bV $.
\end{lemma}
\begin{proof}
The case $j=0$ is covered by  \cite[Lemma 4.1]{GQ14}. {\color{black}For second case $j=1$ we want to show that }
\begin{align}
    b({\bfu_h},{\bfv}-\Pi {\bfv}) &= (\mathcal A\sigma_h,
    \tau-\Pi^{\div,\mathbb S}_{k+2} \tau)  
    + (u_h,\div_\T (\tau-\Pi^{\div,\mathbb S}_{k+2} \tau))
    + (\sigma_h, \nabla_\T (v-\Pi^{\mathrm{grad}}_{k+d} v)) +(\sigma_h,q-\Pi^k q)\notag \\& \quad
    - \langle \hat u_h, (\tau-\Pi^{\div,\mathbb S}_{k+2} \tau) \cdot\nu\rangle_{\partial\T} - \langle\hat\sigma_{h},v-\Pi^{\mathrm{grad}}_{k+d} v\rangle_{\partial\T}
\end{align}
{\color{black}is equal to zero.}
Note that only the second term is modified by the change of the space. The commutative property of $\Pi^{\div,\mathbb S}_{k}$, 
the approximation property of $\Pi_k$ and the fact that $u_h$ belongs to $P^{k+1}(\T,\mathbb V)$ yields to
\begin{align*}
    (u_h, \div_\T(\tau -\Pi^{\div,\mathbb S}_{k+2} \tau)) 
    = (u_h, \div_\T\tau - \Pi_{k+1}(\div_\T\tau))
    = 0 \quad \forall \tau \in H(\div,\T;\mathbb S)
\end{align*}
\end{proof}
The existence of the Fortin operator {\color{black}immediately} leads to the quasioptimality result which we state in the next theorem.
\begin{theorem}[Quasioptimality]
    Let $\bu=(\sigma,u,\hat\sigma_{n},\hat u) \in \bU$ be the solution to continuous formulation \eqref{eq:dpg}. For $k\geq 0$ and $j=0,1$ let  
    $\bu_h=(\sigma_h,u_h,\hat\sigma_{n,h},\hat u_h) \in \bU_h^{k+j}$ be the solution to the discrete DPG formulation \eqref{eq:dpgh}.
   % \begin{align*}
%        b((\sigma,u,\hat\sigma_n,\hat u),(\tau,v,q)) = l(v)\quad \forall (\tau,v,q)\in V
%    \end{align*}
%    has a unique solution $(\sigma,u,\hat\sigma_n,\hat u) \in U$ and practical formulation
  %  \begin{align*}
  %      b((\sigma_h,u_h,\hat\sigma_{n,h},\hat u_h),(\tau_h,v_h,q_h)) = l(v_h)\quad \forall (\tau_h,v_h,q_h)\in V_h
  %  \end{align*}
   % has a unique solution $(\sigma_h,u_h,\hat\sigma_{n,h},\hat u_h) \in U_h$. Moreover
   Then, the following quasioptimality result holds
    \begin{align*}
        \Vert\sigma-\sigma_h\Vert_{L^2(\Omega)}
        &+\Vert u-u_h\Vert_{L^2(\Omega)}
        +\Vert\hat\sigma_n-\hat\sigma_{n,h}\Vert_{H^{-1/2}(\partial\T)}
        +\Vert \hat u-\hat u_h\Vert_{H^{1/2}(\partial\T)} \\
        & \leq C \min_{\rho_h,w_h,\hat\rho_h,\hat w_h\in \bU_h} 
        \Vert\sigma-\rho_h\Vert_{L^2(\Omega)}
        +\Vert u-w_h\Vert_{L^2(\Omega)}
        +\Vert\hat\sigma_n-\hat\rho_h\Vert_{H^{-1/2}(\partial\T)}
        +\Vert \hat u-\hat w_h\Vert_{H^{1/2}(\partial\T)} .
    \end{align*} 
\end{theorem}
{\color{black}The proof follows from the quasioptimality result \cite[Theorem 4.2]{GQ14}.}

\section{Improved a priori convergence}\label{sec:apriori}
This section aims at improving the existing a priori analysis from \cite{BDGQ12} to require only minimal regularity of the solution $u$ to obtain optimal convergence rates.
In fact, in the contribution \cite{BDGQ12} the best approximation error in the third component (for the traces of the stress tensor)
is bounded by the best approximation error of the {\color{black}exact stress} in the $H(\div,\mathbb M)$-norm,
{\color{black}
approximated in the space 
\begin{align}
    RT^{k}_{\Gamma_N}(\T) = \{
    \tau \in H_{\Gamma_N}(\div, \Omega;\mathbb M)
    \ : \ \tau_{|_T} = 
    \tau_1+
    \tau_2 {\mathbf{x}}^\top, 
    \tau_1 \in P^k(\T;\mathbb M), \ 
    \tau_2 \in P^k(\T;\mathbb V)
    \ \forall T \in \mathcal T
    \} \ .
\end{align}
}
To estimate our solution we now introduce the corresponding interpolants.  
\begin{definition}
Let $\Pi^{\div}_{k} : H_{\Gamma_N}(\div, \Omega;\mathbb M) \cap {\color{black}H^{r+1}(\T;\mathbb M)} \rightarrow
RT^{k}_{\Gamma_N}(\T)$ denote the {\color{black}componentwise} Raviart-Thomas interpolant such that
\begin{align*}
    \Vert \tau - \Pi^{\div}_{k} \tau\Vert \leq C_k h^{r+1} \vert\tau\vert_{H^{r+1}(\T)} \quad \mathrm{for}\ r \in [0,k]
\end{align*}
and $\div \Pi^{\div}_{k} \tau = \Pi_k \div\tau$
holds for all $\tau \in H(\div, \Omega;\mathbb M) \cap H^1(\T)$ (see
\cite[Theorem 3]{MR0483555} or \cite[Theorem 16.4]{MR4242224}).
Let $\Pi^\nabla_{k+1}:H^1_{\Gamma_D}(\Omega) \rightarrow S_{\Gamma_D}^{k+1}(\T)$ denote the Scott-Zhang interpolant \cite{SZ90} such 
that 
\begin{align}
    \Vert v - \Pi^\nabla_{k+1} v \Vert \leq C_{k+1} h^{k+1} \Vert v\Vert_{H^{k+2}(\Omega)}. \label{eq:scott-zhang}
\end{align}
\end{definition}
The error estimate of the $H(\div)$ trace, therefore, relies on the term  $||(1 - \Pi_k^{\div} )
\div \sigma ||$. For it to converge when the mesh size reduces, some regularity of 
$\text{div }\sigma$ has to be assumed. Due to \eqref{eq:momentum} this means regularity assumptions on $f$.
An alternative way to estimate the best approximation error corresponding to the trace component was presented in \cite{F18}. 
There the commutativity of $\Pi^{\div}_k$ and its approximation property lead directly to an estimate for $\Vert\hat\sigma_n\Vert_{H^{-1/2}(\partial \T)}$
without considering the original $||(1 - \Pi_k^{\div} )\div \sigma ||$ term. 
In the next theorem, we extend \cite[Theorem 5, Corollary 6]{F18} {\color{black}resp. \cite[Theorem 6]{ultraweak-duality}} to our DPG formulation.
This will be particularly useful to extend our duality argument to more general 
regularity assumptions. {\color{black} Note that the result in \cite[Theorem 5, Corollary 6]{F18} also {\color{black}applies} in the case of reduced regularity.}
\begin{theorem}\label{thm:rates} Let $\bfw = (\chi, w, \gamma_n \chi,
    \gamma_0 w)\in \bU$ with $w \in
    H^{k+2}(\Omega;\mathbb V)$ and $\chi \in H^{k+1}(\T; \mathbb M) \cap
    H(\div,\Omega;\mathbb M)$.  
For $k\in\mathbb N_0$ and $j=0,1$, the best approximation $\bw_h \in \Uh$ satisfies
    \begin{align}\label{improvedregularity}
        %\min_{\bw_h\in \Uh} 
        \Vert \bfw - \bfw_h\Vert_{\bU}
         \leq C h^{k+1} (\Vert w\Vert_{H^{k+2}(\Omega)} 
         + \Vert\chi\Vert_{H^{k+1}(\T)}).
    \end{align}
\end{theorem}
\begin{proof}
     {\color{black}The proof is just the componentwise application of Theorem 6 from \cite{ultraweak-duality}.}
\end{proof}

\section{Distance of $||u-u_h||$ to $U_h^{k,j}$}
With the improved a priori error estimates of the previous section, we now employ a duality argument to show that the error 
$||u-u_h||$ is almost orthogonal to any $g\in U_h^{k,j}$. To this aim, we derive a representation of the solution to the adjoint problem in the following lemma.
\begin{lemma}[Representation of the solution to the adjoint problem]\label{lem:repres} 
Let $g\in L^2(\Omega)$ be and \\ $\bfv := (\tau, v,q)
    \in H_{\Gamma_N}(\div, \Omega;\mathbb S)\times  H^1_{\Gamma_D}(\Omega;\mathbb
    V)\times L^2({\color{black}\Omega ;}\mathbb A)$ is the solution to the adjoint problem
\begin{subequations}
\begin{align}
    \mathcal{A} \tau + \nabla v +q &= 0  &&\mathrm{in}\ \Omega,\\
    \div \tau &= g  && \mathrm{in}\ \Omega, \\
    v &= 0  &&\mathrm{on}\ \Gamma_D, \\
    \tau \cdot \nu &= 0 &&\mathrm{on}\ \Gamma_N.
\end{align}
\label{eq:adj}    
\end{subequations}
Then, there exits a unique element $ \bfw = \Theta^{-1}  \bfv$ of $\bU $ such that 
    \begin{align*}
        &(\bfv, \bfz)_{\bV} = b(\bfw, \bfz)\quad \forall \bfz \in \bV,\\
        &\bfw = (-q,g,-\gamma_n \tau, 0) + (\sigma^*, u^*, \gamma_n \sigma^*, \gamma_0 u^*),
    \end{align*}
    with  $(\sigma^*,
    u^*)\in H_{\Gamma_N}(\div, \Omega;\mathbb M) \times H^1_{\Gamma_D}(\Omega;\mathbb V)$ the solution to the problem 
    \begin{subequations}
    \label{eq:adjointelasticity}
    \begin{align}
        \mathcal A \sigma^* - \varepsilon(u^*)   &= \tau
         && \mathrm{in}\  \Omega,\\ 
        \label{eq:momentums}
        -\div \sigma &= \div \mathcal A\tau + v && \mathrm{in} \  \Omega ,\\
        u &= 0 && \mathrm{on}\  \Gamma_D,\\
        \sigma \cdot \nu &= 0 && \mathrm{on}\  \Gamma_N .
    \end{align}
\end{subequations}
Moreover, it holds
    \begin{align*}
        \Vert v\Vert_{H^2(\Omega)} + \Vert\tau\Vert_{H^1(\T)} + \Vert q\Vert
        + \Vert u^*\Vert_{H^2(\Omega)} + \Vert \sigma^*\Vert_{H^1(\T)} 
        \leq C \Vert g\Vert. 
    \end{align*}
\end{lemma}

\begin{proof}
Since 
 $H_{\Gamma_N}(\div, \Omega;\mathbb S)\times  H^1_{\Gamma_D}(\Omega;\mathbb V) \times
L^2({\color{black}\Omega ;}\mathbb A) \subset H(\div,\T;\mathbb S) \times H^1(\T;\mathbb V) \times L^2({\color{black}\T ;}\mathbb A) = \bV$ the solution $\bfv =(\tau,v,q)$ of the adjoint problem \eqref{eq:adj} lies in $\bV$. 
Any test functions $(\mu,\lambda,\rho)\in \bV$ allow for 
the product
    \begin{align*}
        ((\tau,v,q), (\mu,\lambda,\rho))_{\bV} = (\div \tau, \div_\T\mu) + (\tau, \mu) 
                                    + ( \nabla v, \nabla_\T\lambda) + (v,\lambda) + (q,\rho).
    \end{align*} 
where the scalar product $(\cdot,\cdot)_{\bV}$ induces the $\Vert\cdot\Vert_{\bV}$ norm. Moreover, the adjoint problem implies $\div \tau = g$ and thus
    \begin{align*}
        (\div\tau, \div_\T \mu) = (g, \div_\T \mu) = b((0,g,0,0),(\mu,\lambda,\rho)).
    \end{align*}
    The constitutive equation $\nabla v = - \mathcal{A}\tau - q$ and an integration by parts leads to
    \begin{align*}
        ( \nabla v, \nabla_\T\lambda) &= -(\mathcal A\tau, \varepsilon_\T(\lambda)) - (q,\nabla_\T \lambda) \\
        &=-\langle\gamma_n (\mathcal A \tau), \lambda\rangle_{\partial\T} + (\div \mathcal A\tau, \lambda) - (q,\nabla_\T \lambda)\\
        &= -b(0,0,\gamma_n (\mathcal A\tau),0),(\mu,\lambda,\rho)) + (\div \mathcal A\tau,\lambda) - (q,\nabla_\T \lambda).
    \end{align*}
    Combining the above equations, we conclude 
    \begin{align*}
    ((\tau,v,q), (\mu,\lambda,\rho))_\bV &= (\div \tau, \div_\T\mu) + (\tau, \mu) 
                                    + ( \nabla v, \nabla_\T\lambda) + (v,\lambda)
                                    \underbrace{-(q,\nabla_\T \lambda) - (q,p) - (\mathcal{A} q , \mu)}_{-b((q,0,0,0),(\mu,\lambda,\rho))}\\
                                  &= b((-q,g,-\gamma_n(\mathcal A\tau),0),(\mu,\lambda,\rho)) + (\div\mathcal A\tau +v ,\lambda) + (\mathcal{A}\mathcal{A}^{-1}\tau,\mu)\ ,
    \end{align*}
    since $(\mathcal{A} q , \mu) = 0$.
    For
$
        \bfw:= (-q,g,-\gamma_n (\mathcal A\tau),0) + (\sigma^*, u^*, \gamma_n\sigma^*, \gamma_0u^*)
$
    with $(\sigma^*, u^*)$ solution to  \eqref{eq:adjointelasticity}
    we obtain
    \begin{align*}
        ((\tau,v,q), (\mu,\lambda,\rho))_{\bV} = b(\bfw,(\mu,\lambda,\rho)) \quad \forall (\mu,\lambda,\rho) \in \bV.
    \end{align*}
    Furthermore, the assumptions on the domain leading to  \eqref{eq:regularity} also imply
\begin{align}
\Vert u \Vert_{H^2(\Omega)} + \Vert \sigma\Vert_{H^1(\T)} &
\leq C (\Vert \div{\mathcal A \tau}+v
\Vert 
+ \Vert \tau \Vert_{H^1(\T)}) 
\end{align}
and
\begin{align*}
    \Vert v \Vert_{H^2(\Omega)} + \Vert \tau\Vert_{H^1(\T)} + \Vert q \Vert &\leq C \Vert g\Vert  .
\end{align*}
        Therefore, we conclude 
    \begin{align*}
        \Vert v\Vert_{H^2(\Omega)} + \Vert\tau\Vert_{H^1(\T)} + \Vert q \Vert + \Vert u^*\Vert_{H^2(\Omega)} + \Vert \sigma^*\Vert_{H^1(\T)} \leq C \Vert g\Vert.
    \end{align*}
\end{proof}
In order to estimate the distance of the error to $U_h^{k,j}$, we will need to proof in Lemma \ref{lem:h+} that
the representation formula of the above lemma implies  an orthogonality relation of the type
$$b(\bfu -\bfu_h ,\bfv) = (u-u_h,g) + \hat\rho(\bfu -\bfu_h,\bv)$$
for any $\bv \in \bV $
where $\hat\rho(\bfu -\bfu_h,\bv)
=
- \langle \hat u-\hat u_h, \tau\cdot\nu\rangle_{\partial\T}
                                - \langle\hat\sigma_{n} -\hat\sigma_{n,h},v\rangle_{\partial\T}
$ vanishes. For the convenience of the reader, we therefore recall the crucial property of the trace-spaces before proceeding with Lemma \ref{lem:h+} and the estimation of the distance of the error to $U_h^{k,j}$. 

\begin{lemma}[\text{\cite[Lemma 3.1]{KFD16}}]\label{lem:traces}
    Let $\Gamma_D$ and $\Gamma_N$ be relative open subsets in $\partial \Omega$,
    such that $\partial \Omega = \overline{\Gamma_D \cup \Gamma_N}$ and
    $\Gamma_D \cap \Gamma_N = \emptyset$.
    \begin{enumerate}
        \item Let $v\in H^1(\T)$. Then $v\in H^1_{\Gamma_D}(\Omega)$ iff
        $\langle \hat\tau,v\rangle_{\partial\T} =0\ \forall \hat\tau \in
        H^{-1/2}_{\Gamma_N}(\partial\T)$.
        \item Let $\tau \in H(\div,\T)$. Then $\tau \in H_{\Gamma_N}(\div,\Omega)$ iff 
        $\langle \hat u, \tau\cdot\nu\rangle_{\partial\T}=0 \ \forall \hat u\in H^{1/2}_{\Gamma_D} (\partial\T)$.
    \end{enumerate}
\end{lemma}

We are now in place to prove the crucial error estimate $|(u-u_h,g)|\leq C h \Vert\bfu-\bfu_h\Vert_ {\bU} \Vert g\Vert$ for any $g \in U_h^{k,j}$ which is the essence of the next lemma. The proof relies on the equivalence of the DPG formulation \eqref{eq:dpgh} with the following mixed formulation: find $(\bfepsilon_h,\bfu_h) \in
    V_h\times U_h$ such that
\begin{align}
\label{eq:dpgmixedh}
    a((\bfepsilon_h,\bfu_h);(\bfv_h,\bfw_h))  &= l(\bfv_h)   &&\forall \bfv_h\in V_h\ \forall\bfw_h\in U_h
\end{align}
with $a((\bfepsilon_h,\bfu_h);(\bfv_h,\bfw_h)) = (\bfepsilon_h,\bfv_h)_\bV +
b(\bfu_h,\bfv_h) - b(\bfw_h,\bfepsilon_h) $. The continuous counterpart reads 
\begin{align}
\label{eq:dpgmixed}
a((\bfepsilon, \bfu);(\bfv, \bfw)) = l(\bfv)\ {\color{black}\forall} \bfv\in\bV \ {\color{black} \forall}  \bfw\in \bU,
\end{align}
where $\bfu \in \bU$ solves \eqref{eq:dpg-abstract} and $\bfepsilon = 0$.
Recall from \cite{CDG14} that  $\bfepsilon_h \in \bV_h$ provide a reliable error
    estimator, i.e. 
    \begin{align}\label{eq:epsilonhiserrorestimator}
        \Vert \bfepsilon_h\Vert_\bV \lesssim \Vert
    \bfu-\bfu_h\Vert_\bU .  \end{align}

\begin{lemma}
[Distance of the error to $U_h^{k,j}$]\label{lem:h+}
For $k \geq 0$ and $j=0,1$, let $\bfu\in \bU$ be the solution to the ultra-weak formulation \eqref{eq:dpg} and
    $\bfu_h \in \Uh $ be the solution 
    of the discrete problem \eqref{eq:dpgh}. 
    Then, there exists a constant $C>0$ independent of $h$, such that
    \begin{align*}
        |(u-u_h,g)|\leq C h \Vert\bfu-\bfu_h\Vert_ {\bU} \Vert g\Vert 
    \end{align*}
    for any $g \in U_h^{k,j}$.
\end{lemma}
\begin{proof}
For $g \in U_h^{k,j} \subset L^2(\Omega;\mathbb V)$, let $\bfv = (\tau,v,q)\in \bV$ denote the solution to the adjoint problem \eqref{eq:adj} with data $g$, i.e.
\begin{align*}
        b(\bfu -\bfu_h ,\bfv) &=
        (\sigma -\sigma_h,\mathcal{A}\tau)
                                + (\sigma-\sigma_h,\nabla v)
                                + (\sigma-\sigma_h,q) 
                                +(u-u_h,\div\tau)
                                + \hat\rho(\bfu -\bfu_h,\bv).
    \end{align*}
Since Lemma \ref{lem:traces} implies that $\hat\rho(\bfu -\bfu_h,\bv)$ vanishes, the equations $\div \tau=g$ and the constitutive equation $\mathcal A \tau + \nabla v +q = 0$ from the adjoint problem \eqref{eq:adj} lead to
\begin{align*}
        b(\bfu -\bfu_h ,\bfv) 
        &= (u-u_h,g) .
    \end{align*}
{\color{black}The property $b(\bfw,\tilde \bfv) = (\bfv,\tilde\bfv)_{\bV} = (\tilde\bfv,\bfv)_{\bV}$ {\color{black}for all} $\tilde{\bfv}\in \bV$ and $\bfw = \Theta^{-1}  \bfv$ from 
Lemma \ref{lem:repres}, imply   
\begin{align*}
  (u-u_h,g)  &=       b(\bfu -\bfu_h ,\bfv) \\
&=(\bfepsilon-\bfepsilon_h,\bfv)_{\bV} + b(\bfu-\bfu_h,\bfv) -
        (\bfv,\bfepsilon-\bfepsilon_h)_{\bV}\\
        &= a((\bfu-\bfu_h, \bfepsilon -\bfepsilon_h),(\bfw,\bfv)).
    \end{align*}}
The Galerkin orthogonality shows that any $(\bfw_h,\bfv_h)\in \Uh\times V_h$ statisfies
\begin{align*}
  (u-u_h,g)  &= a((\bfu-\bfu_h, \bfepsilon -\bfepsilon_h),(\bfw -\bfw_h,\bfv -\bfv_h)) .
  \end{align*}
Consequently, the {\color{black}boundedness} of $a$ and \eqref{eq:epsilonhiserrorestimator} imply 
    \begin{align*}
        (u-u_h,g) 
                  &\lesssim \Vert \bfu -\bfu_h \Vert_\bU (\Vert \bfw -\bfw_h\Vert_\bU + \Vert \bfv-\bfv_h\Vert_\bV) .
    \end{align*}
We now estimated the two terms in the bracket on the right-hand side separately.
    \begin{itemize}
        \item {\bf Estimation of $\Vert \bfv -\bfv_h\Vert_\bV$, j=0
        }
        
  For $\bfv_h = {\color{black}(\Pi_{k+2}^{\div, \mathbb S}\tau, \Pi_1^\nabla v},\Pi_0 q) \in \bV_h$, the approximation properties of the projections  \eqref{eq:div-sym}, \eqref{eq:scott-zhang} and Lemma \ref{lem:repres} imply   
    \begin{align*}
        \Vert \bfv-\bfv_h\Vert_{\bV} &\leq \Vert v 
        - \Pi^\nabla_1 v\Vert_{H^1(\Omega)} + \Vert \tau 
        - \Pi^{\div,\mathbb S}_{{\color{black}k+2}} \tau\Vert_{H(\div,\Omega)} +\Vert q -\Pi_0 q\Vert \\
        &\lesssim h \Vert g \Vert + \Vert\div(\tau - \Pi^{\div,\mathbb S}_{\color{black}k+2} \tau)\Vert + \Vert q - \Pi_0 q \Vert.
    \end{align*}
    Moreover, the commutativity property {\color{black}\eqref{eq:commu_Pi_div}} of $\Pi^{\div, \mathbb S}_{\color{black}k+2}$, the adjoint
    problem and $g \in P^k(\T)$ lead to
    \begin{align*}
        \Vert\div(\tau - \Pi^{\div,\mathbb S}_{\color{black}k+2} \tau)\Vert = \Vert (1-\Pi_{\color{black}k+1})\div\tau\Vert = \Vert (1-\Pi_{\color{black}k+1}) g\Vert = 0. 
    \end{align*} 
For the last term, we infer
    \begin{align*}
        \Vert \Pi_0 q -q \Vert 
        &= \Vert \Pi_0 \nabla v -\nabla v + \Pi_0 \mathcal{A}\tau - \mathcal A \tau\Vert\\
        &\leq \Vert \Pi_0 \nabla v -\nabla v\Vert + \Vert\Pi_0 \mathcal{A}\tau - \mathcal A \tau\Vert\\
        &\lesssim h \Vert v\Vert_{H^2(\Omega)} + h\Vert \tau\Vert_{H^1(\T)} \lesssim h \Vert g\Vert.
    \end{align*}
    Overall, we obtain
        \begin{align}
    \label{eq:h+1}\Vert\bfv-\bfv_h\Vert_\bV\lesssim h\Vert g\Vert .
    \end{align}
    \item {\bf Estimation of $\Vert \bfv -\bfv_h\Vert_\bV$, j=1
        }
    We follow the lines of the proof of the case $j=0$ choosing $\bfv_h =({\color{black}\Pi_{k+2}^{\div,\mathbb S}\tau,\Pi_1^\nabla v},\Pi_0 q) \in \bV_h$. This leads to 
    \begin{align}
    \label{eq:h+2}
        \Vert\bfv-\bfv_h\Vert_\bV\lesssim h\Vert g\Vert .
    \end{align}
    \item {\bf Estimation of $\Vert \bfw -\bfw_h\Vert_\bU$
        }\\
     Considering the representation $\bfw =
    (-q,g,-\gamma_n\tau,0) + \tilde\bfw$ from Lemma \ref{lem:repres},  we choose 
    $$\bfw_h = (-\Pi_0 q,g,-\gamma_n \Pi_0^{\div} \tau, {\color{black}0}) + \tilde\bfw_h  ,$$ 
    {\color{black}whereby} $\tilde \bfw_h=(\Pi_k w, \Pi_k \chi,\gamma_0 (\Pi_{k+1}^\nabla w), \gamma_n(\Pi^{\div}_{k} \chi))$ {\color{black}denotes}
    best approximation and lead to 
    \begin{align}
        \Vert\bfw-\bfw_h\Vert_\bU \leq\Vert q - \Pi_0 q  \Vert +\Vert\gamma_n(\tau-\Pi^{\div}_0 \tau)\Vert_{H^{-1/2}(\partial\T)} + \Vert\tilde\bfw-\tilde\bfw_h\Vert_\bU  .
    \end{align}
    The first two terms can be estimate using the approximation property of $\Pi_0 q$ in $L^2(\Omega)$ and
    $\gamma_n\Pi^{\div}_k$ in $H^{-1/2}(\partial\T)$. Lemma 4 and Theorem 3 imply
        \begin{align}
        \Vert \tilde{\bfw} - \tilde{\bfw}_h \Vert_{\bU} \lesssim h (\Vert u^* \Vert_{H^2(\Omega)} + \Vert\sigma^*\Vert_{H^1(\T)})\leq h \Vert g\Vert
    \end{align}
    and thus
    \begin{align}
        \label{eq:h+3}
        \Vert {\bfw} - {\bfw}_h \Vert_{\bU} \lesssim  h \Vert g\Vert  .
    \end{align}
    \end{itemize}
    The estimates \eqref{eq:h+1} (respectively \eqref{eq:h+2}) and \eqref{eq:h+3} lead to
    \begin{align*}
        |(u-u_h,g)| \leq C h\Vert \bfu-\bfu_h\Vert_\bU\Vert g\Vert
    \end{align*}
    in the case $j=0$ (respectively $j=1$). 
\end{proof}
\section{Superconvergence through increased polynomial degree and postprocessing}\label{sec:superconvergence}
The previous section shows that the error $||u-u_h||$ is almost orthogonal to any $g\in \bU_h^{k,0}$. 
As an auxiliary result, we also obtain the supercloseness of $u_h$ to the $L^2$ projection $\Pi_k u$, as 
stated in the following theorem. {\color{black}This is a vector valued generalisation from the scalar result in \cite[Theorem 3]{ultraweak-duality} and is included here for the convenience of the reader.}
\begin{theorem}[Supercloseness to $L^2$ projection] \label{thm:supercloseness}
    Let $\bfu = (\sigma,u,\hat\sigma_n,\hat u)\in \bU$ the solution to the
    ultra-weak formulation \eqref{eq:dpg} and assume that $u \in H^{k+2}(\Omega;\mathbb V)$ and
    $\sigma \in H^{k+1}(\T;\mathbb M)$. Let $\bfu_h =
    (\sigma_h,u_h,\hat\sigma_n,\hat u)\in \bU_h^{k,0}$ the solution to the ultra-weak
    formulation \eqref{eq:dpgh}, then it holds 
    \begin{align*}
    \Vert u_h -\Pi_k u\Vert &\leq C h^{k+2}(\Vert u\Vert_{H^{k+2}(\Omega)} + \Vert\sigma\Vert_{H^{k+1}(\T)}), \\
    \Vert u-\Pi_ku\Vert &\leq \Vert u-u_h\Vert\leq \Vert u-\Pi_ku\Vert + C h^{k+2}(\Vert u\Vert_{H^{k+2}(\Omega)} + \Vert\sigma\Vert_{H^{k+1}(\T)}).
    \end{align*}
\end{theorem}
\begin{proof}
    The triangle inequality and the properties of $\Pi_k$ lead to 
    \begin{align*}
        \Vert u-\Pi_ku\Vert \leq \Vert u-u_h\Vert\leq \Vert u-\Pi_ku\Vert + \Vert\Pi_k(u-u_h)\Vert.
    \end{align*}
    Choosing $g:= \Pi_k u -u_h \in P^k(\T,\mathbb V) = U_h^{k,0}$ leads to
    \begin{align*}
        \Vert g \Vert^2 = (g,g) = (\Pi_k(u-u_h),g) = (u-u_h,g).
    \end{align*}
    {\color{black}Lemma \ref{lem:h+} implies}
    \begin{align*}
        \Vert g\Vert^2 = (u-u_h,g) \lesssim h \Vert\bfu-\bfu_h\Vert_\bU\Vert g\Vert \lesssim h^{k+2}(\Vert u\Vert_{H^{k+2}(\Omega)} + \Vert\sigma\Vert_{H^{k+1}(\T)})\Vert g\Vert.
    \end{align*}
    Dividing by $\Vert g\Vert$ concludes the proof.
\end{proof}

The a priori estimates of section \ref{sec:apriori} prove the same order of convergence in all variables, i.e.
for the stress and the displacements. However, the point of using an augmented trial space (i.e. the finite
element space $\Uh$ with $j=1$) is to obtain an improved convergence rate for the displacement field. The aim of
this section is therefore to prove that in the case $j=1$ the error in the displacement field $\|u-u_h\|$
converges at a higher rate than the total error. To achieve this, we use the improved convergence rate from
Section \ref{sec:apriori} and combine it with the result of our duality argument from Lemma \ref{lem:h+}.
{\color{black}The following proofs use similar arguments as the proof of Theorem \ref{thm:supercloseness}.}
\begin{theorem}[Improved convergence rate for $j=1$] \label{thm:superconvergence} For  $f\in L^2(\Omega;\mathbb V)$, let $\bfu
    =(\sigma,u,\hat\sigma_n,\hat u)\in \bU$ be the solution to the ultra-weak
    formulation \eqref{eq:dpgmixed}. Let $\bfu_h \in  \bU_h^{k,1}$ the solution to the
    corresponding discrete problem \eqref{eq:dpgmixedh}. 
    %Assume also that
    %$RT^{k+1}_{\Gamma_N}(\T;\mathbb M)\times S^1_{\Gamma_D}(\T;\mathbb V)\times  P_0(\T) \subseteq V_h$, 
    If $u\in H^{k+2}(\Omega;\mathbb V)$ and $\sigma\in
    H^{k+1}(\T;\mathbb M)$, then 
    \begin{align*}
        \Vert u-u_h\Vert \leq C h^{k+2} (\Vert u \Vert_{H^{k+2}(\Omega)} + \Vert\sigma\Vert_{H^{k+1}(\T)})
    \end{align*}
    holds.
\end{theorem}
\begin{proof}
    Using the triangle inequality we obtain
    \begin{align}
    \label{eq:superconvergencetri}
        \Vert u-u_h\Vert\leq\Vert u-\Pi_{k+1}u\Vert+\Vert\Pi_{k+1} u- u_h\Vert 
        \leq\Vert u-\Pi_{k+1}u\Vert+\Vert g\Vert 
    \end{align}
    with $g:= \Pi_{k+1}u-u_h \in P^{k+1}(\T) = U_h^{k,1}$. 
The first term can
    be estimated using the approximation property of $\Pi_{k+1}$, i.e.
    \begin{align}
    \label{eq:superconvergence1}
        \Vert u -\Pi_{k+1}u\Vert\lesssim h^{k+2}\Vert u\Vert_{H^{k+2}(\Omega)}.
    \end{align}    
    Moreover, since
    \begin{align*}
        \Vert g \Vert^2 =(g,g) = (\Pi_{k+1} (u-u_h),g) = (u-u_h,g),
    \end{align*}
    Lemma \ref{lem:h+} and Theorem \ref{thm:rates} lead to
    \begin{align*}
        \Vert g\Vert^2 &=(u-u_h,g) \lesssim h \Vert \bfu-\bfu_h\Vert_{U} \Vert g\Vert
        \lesssim h h^{k+1} (\Vert u\Vert_{H^{k+2}(\Omega)} + \Vert\sigma\Vert_{H^{k+1}(\T)})\Vert g\Vert 
    \end{align*}
    and thus to
    \begin{align}
        \label{eq:superconvergence2}
\Vert g\Vert \lesssim h^{k+2} (\Vert
    u\Vert_{H^{k+2}(\Omega)} + \Vert \sigma\Vert_{H^{k+1}(\T)}) . \end{align}  
Inserting \eqref{eq:superconvergence2} and \eqref{eq:superconvergence1} in \eqref{eq:superconvergencetri} finishes the proof.
\end{proof}

Another possibility to achieve higher convergence rates is to postprocess the part of the solution $u_h$. To this aim, we 
introduce the space of rigid body motion $\RM(\T)$, defined as the kernel 
of the symmetric gradient
\begin{align*}
    \RM(\T) = \{v \in H^1(\T;{\color{black}\mathbb V})\ |\ \varepsilon_T(v)=0\ \forall T \in\T \} \subset P^1(\T)
\end{align*}
and let $\Pi_{\mathrm{rm}}$ denotes the $L^2$-projection onto $\RM({\color{black}\T})$. For the convenience of the reader, we state the well-known approximation property of this projection in the next remark.
\begin{lemma}
    The projection $\Pi_{\mathrm{rm}}$ fullfils the following approximation property 
    \begin{align}\label{eq:app_rm}
        \Vert \Pi_{\mathrm{rm}} v - v \Vert \leq C h \Vert \varepsilon_\T(v)\Vert 
    \end{align}
    for all $v\in H^1(\T)$.  
\end{lemma}
\begin{proof}
Let us first note that the left-hand side vanishes if the right-hand side is zero. In fact, if $\varepsilon_\T(v) = 0$,  $\Pi_{\mathrm{rm}} v = 0$ and the inequality becomes trivial.
Consider now $v\in H^1(\T) / \RM(\T)$. Then,
        \begin{align*}
            (\Pi_{\mathrm{rm}} v , w) = (\Pi_0 v ,w ) \quad \forall w \in P_0(\T).
        \end{align*}
        Therefore, $\Pi_{\mathrm{rm}}$ inherits the approximation 
        property of $\Pi_0$ and together with the Korn inequality, it holds
        \begin{align*}
            \Vert \Pi_{\mathrm{rm}} v -v \Vert \leq C h \Vert \nabla_\T v\Vert \leq C h \Vert \varepsilon_\T (v)\Vert .
        \end{align*}
\end{proof}
In order to post-process the part $u_h$ of the solution $\bfu_h = (\sigma_h,u_h,\hat\sigma_n,\hat u)\in U_h$ of the
ultra-weak formulation, we consider a local Neumann problem in the spirit of \cite{S88}:\\ Find $\tilde u_h
\in P^{k+1}(\T;\mathbb V)$ such that
\begin{align}\label{eq:post}
\begin{cases}
&\Pi_{\mathrm{rm}}
\tilde{u}_h = \Pi_{\mathrm{rm}} u_h, \\
 &   (\varepsilon(\tilde{u}_h),\varepsilon(v_h))  = ( \mathcal{A}\sigma_h,\varepsilon(v_h))
    \quad \forall v_h \in \{v \in P^{k+1}(\T;\mathbb V)\ |\ (v,w)_T = 0\ {\color{black} \forall T \in \T}, \ \forall w \in \RM(\T)\} .
\end{cases}
\end{align}

\begin{theorem}(Convergence rate of the post-processing)\label{thm:postprocessing}
    Let $\bfu =(\sigma, u, \hat\sigma_n,\hat u) \in \bU$ be the solution to the
    ultra-weak formulation \eqref{eq:dpgmixed} for some $f\in L^2(\Omega;\mathbb V)$, and assume that $u \in H^{k+2}(\Omega;\mathbb V)$ and
    $\sigma \in H^{k+1}(\T;\mathbb M)$. Let $\bfu_h =
    (\sigma_h,u_h,\hat\sigma_n,\hat u)\in \bU_h^{k,0}$ the solution to the ultra-weak
    formulation and $\tilde{u}_h\in P^{k+1}(\T;\mathbb V)$ the post-processing satisfying \eqref{eq:post}.
    Then, for $k\geq 1$
    \begin{align*}
        \Vert u_h-\tilde{u}_h\Vert \leq C h^{k+2} (\Vert u\Vert_{H^{k+2}(\Omega)} + \Vert\sigma\Vert_{H^{k+1}(\T)})
    \end{align*}
     holds.
\end{theorem}
\begin{proof}
     $\Pi_{\mathrm{rm}} \tilde{u}_h =\Pi_{\mathrm{rm}} u_h$
    and the approximation properties \eqref{eq:app_rm} of $\Pi_{\mathrm{rm}}$ lead to
    \begin{align*}
        \Vert u - \tilde{u}_h\Vert \leq \Vert(1-\Pi_{\mathrm{rm}})(u-\tilde{u}_h)\Vert + 
        \Vert \Pi_{\mathrm{rm}}(u-\tilde{u}_h)\Vert \lesssim h \Vert \varepsilon_\T (u-\tilde{u}_h)\Vert + 
        \Vert \Pi_{\mathrm{rm}} (u-u_h)\Vert .
    \end{align*}
     Moreover, $g := \Pi_{\mathrm{rm}}(u-u_h) \in \RM(\T) \subset P^1(\T)$ 
     Lemma \ref{lem:h+} and Theorem \ref{thm:rates} imply
    \begin{align*}
        \Vert g \Vert^2 = (\Pi_{\mathrm{rm}}(u-u_h),g) = (u-u_h,g) \lesssim h \Vert \bfu-\bfu_h\Vert_\bU \Vert g\Vert 
        \lesssim h^{k+2} (\Vert u \Vert_{H^{k+2}(\Omega)} + \Vert\sigma\Vert_{H^{k+1}(\T)})\Vert g\Vert .
    \end{align*}
    It remains to estimate $\Vert\varepsilon_\T(u-\tilde u_h)\Vert$. Since
$\bar u_h\in P^{k+1}(\T;\mathbb V)$
satisfies
    \begin{align*}
        (\varepsilon(\bar u_h),\varepsilon(v_h)) &= ( \mathcal{A}\sigma,\varepsilon (v_h))_T
        &&\forall v_h\in \{v \in P^{k+1}(\T;\mathbb V)\ |\ (v,w)_T = 0\ {\color{black}\forall T \in \T},\ \forall w \in \RM(\T)\} ,
    \end{align*}
it holds
    \begin{align*}
        \Vert \varepsilon_\T(\bar u_h &-\tilde{u}_h)\Vert^2 = (-\mathcal{A}(\sigma-\sigma_h),\varepsilon_\T (\bar u_h-\tilde{u}_h)) \\
        &\lesssim \Vert \bfu - \bfu_h \Vert_\bU \Vert\varepsilon_\T(\bar u_h-\tilde u_h)\Vert \lesssim h^{k+1} (\Vert u \Vert_{H^{k+2}(\Omega)} 
        + \Vert\sigma\Vert_{H^{k+1}(\T)} ) \Vert\varepsilon_\T(\bar u_h-\tilde u_h)\Vert .
    \end{align*}
    Moreover, the Galerkin orthogonality $(\varepsilon_\T (u-\bar
    u_h),\varepsilon_\T(v_h)) = 0$ for all $v_h\in P^{k+1}(\T;\mathbb V)$ leads to
    \begin{align*}
        \Vert\varepsilon_\T(u-\bar u_h)\Vert  
        = \min_{v_h\in P^{k+1}(\T)} \Vert\varepsilon_\T(u- v_h)\Vert 
        \lesssim h^{k+1} \Vert u \Vert_{H^{k+2}(\Omega)}.
    \end{align*}
    Combining the previous inequalities finishes the proof:
    \begin{align*}
        \Vert u-\tilde u_h \Vert 
        &\lesssim  h \Vert \varepsilon_\T (u-\tilde{u}_h)\Vert + \Vert g\Vert \\
        &\lesssim h (\Vert\varepsilon_\T(u-\bar u_h)\Vert + \Vert\varepsilon_\T(\bar u_h-\tilde u_h)\Vert)   + h^{k+2} (\Vert u \Vert_{H^{k+2}(\Omega)} + \Vert\sigma\Vert_{H^{k+1}(\T)})\\
        &\lesssim h^{k+2} (\Vert u \Vert_{H^{k+2}(\Omega)} + \Vert\sigma\Vert_{H^{k+1}(\T)}).
    \end{align*}
\end{proof}
{\color{black}
\begin{remark}[Regularity of the solution]
    For the purpose of the exposition, we assumed previously $u\in H^2(\Omega)$. This can be not expected in general. {\color{black} 
    For example in general domains, let the solution
     $u \in H^{1+s}(\Omega)$ for some $s\in (1/2,k+1]$ and the adjoint solution $v \in H^{1+s'}(\Omega)$ for some $s'\in(1/2,1]$.
    In the presence of this different 
    regularity assumption, Theorems \ref{thm:supercloseness}, \ref{thm:superconvergence} and \ref{thm:postprocessing} remain valid by
    replacing $h^{k+2}$ by $h^{1+s+s'}$.
    }
\end{remark}
}
\phantom{.}\hspace{-2cm}\includegraphics[width=\paperwidth]{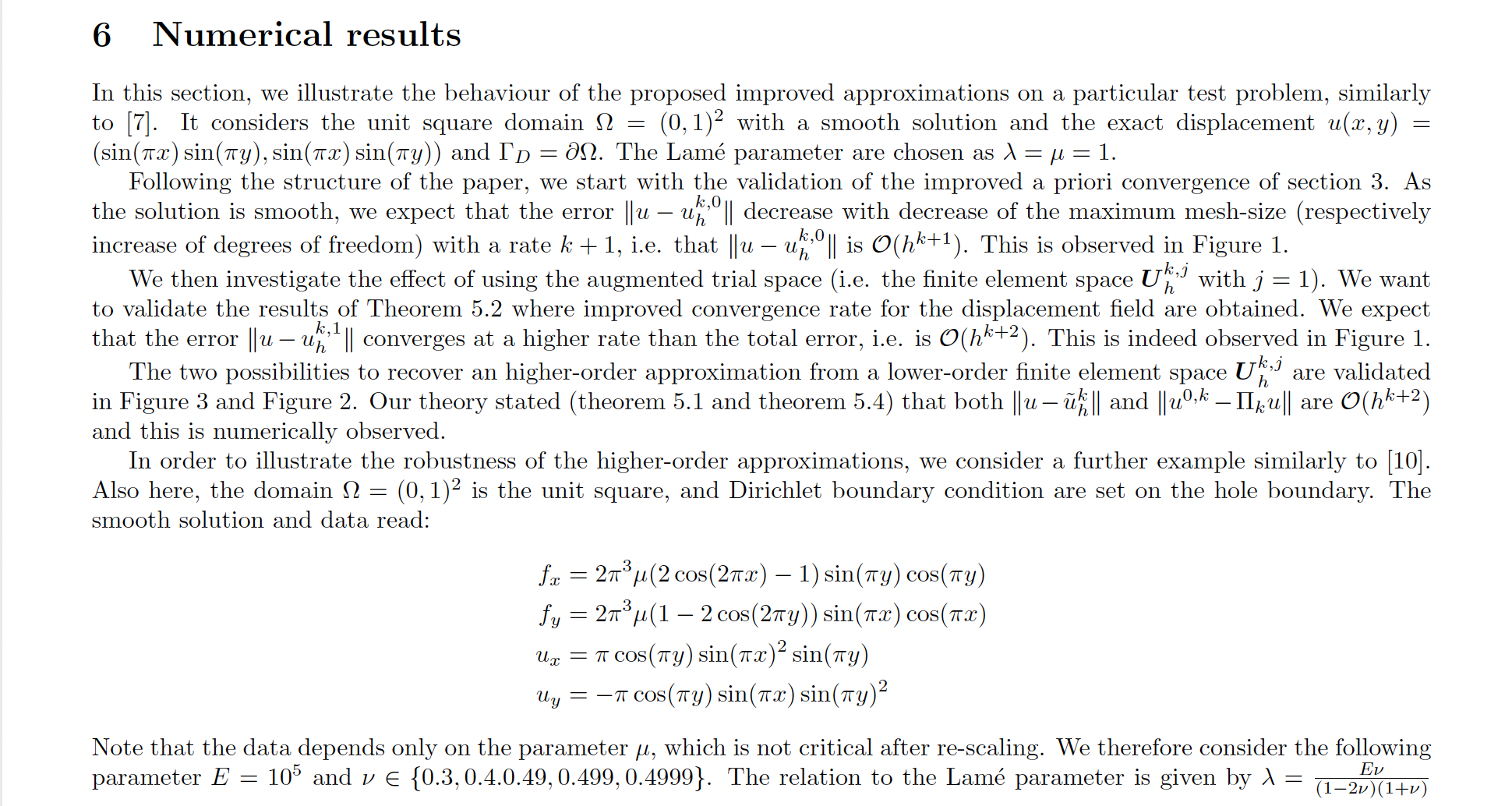}
\newpage


\section*{Supplementary material}

and $\mu = \frac{E}{2(1+\nu)}$. We observe in Figures 4 and 5 that the error bounds for the different values $\nu$ are robust. In order to emphasize that our scheme is looking free, we added in Figure 4 the error for the solution $u_h$ of standard displacement formulation with continuous piecewise linear polynomials. The $L^2$-projection is independent of the Lamé-parameter such that the corresponding lines distance in Figure 5 the same since it is independent on them and they are not distinguishable from the results of .

A second numerical test is performed on the L-shape domain $\Omega = (-1,1)^2 \setminus [0,1] \times [0,-1]$, see [11]. We set $u = 0$ on the whole boundary $\partial\Omega$, $E = 1$ and consider the parameter $\nu \in \{0.4, 0.49, 0.499, 0.4999, 0.49999\}$. The right-hand side $f$ is defined by

$$f = \begin{cases} (1,0)^T & \text{if } xy \geq 0, \ \max\{|x|,|y|\} \leq 0.5, \\ (0,0)^T & \text{else} \end{cases}.$$

A uniform refinement strategy leads to robust approximations with respect to the incompressibility parameter, but as expected achieve improved convergence rates neither for the augmented trial space nor for the post-processing, see Figures 7 and 6.

In orter to retain optimal convergence rates, the classical AFEM loop with Dörfler-Marking and bulk parameter $\theta = 0.5$ is used to obtain the results for the augmented trial spaces and for the post-processing presented in Figure 8 and in Figure 9. The error estimator is the residual $\varepsilon$, it is proven to be reliable and efficient [8]. As expected, optimal convergence rates for the trial space $u_h^{1,0}$ is recovered, as well as for $u_h^{1,1}$ and $\tilde{u}_h^1$. Not surprisingly, the mesh refinement is concentrated in the re-entrant corner, see Figure 10. All results show that the numerical experiment is robust for $\nu \to 0.5$.

Figure 1: Convergence of the DPG solution for augmented trial spaces

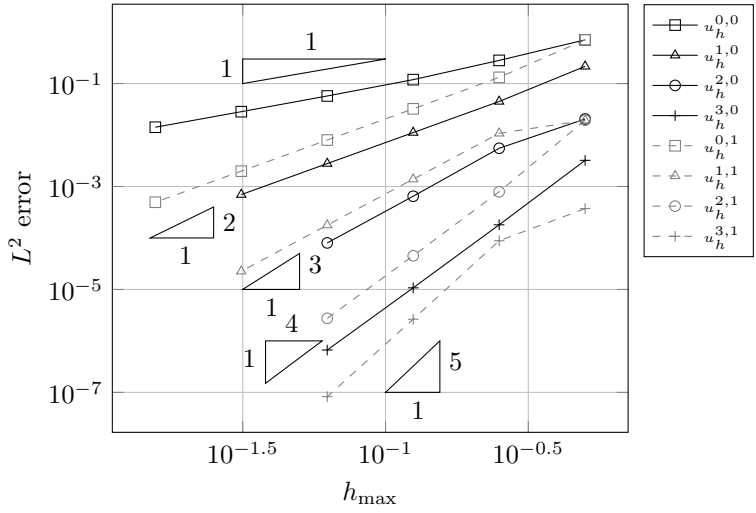

Figure 2: Convergence of the post-processing

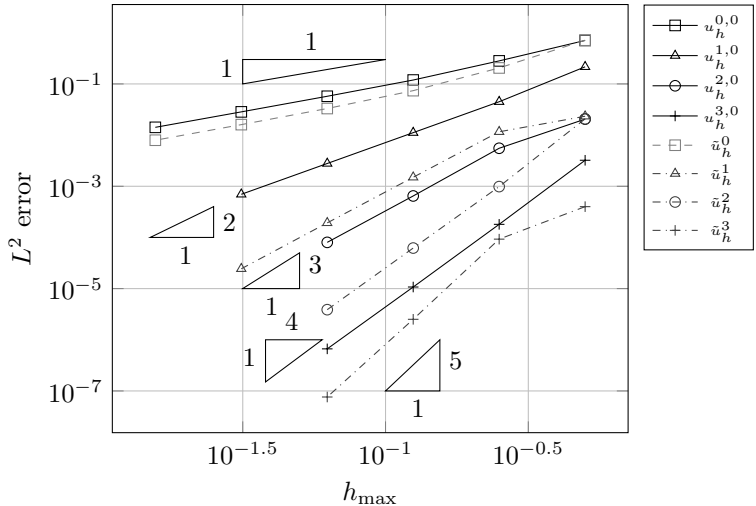

Figure 3: Superconvergence of the $L^2$ projection

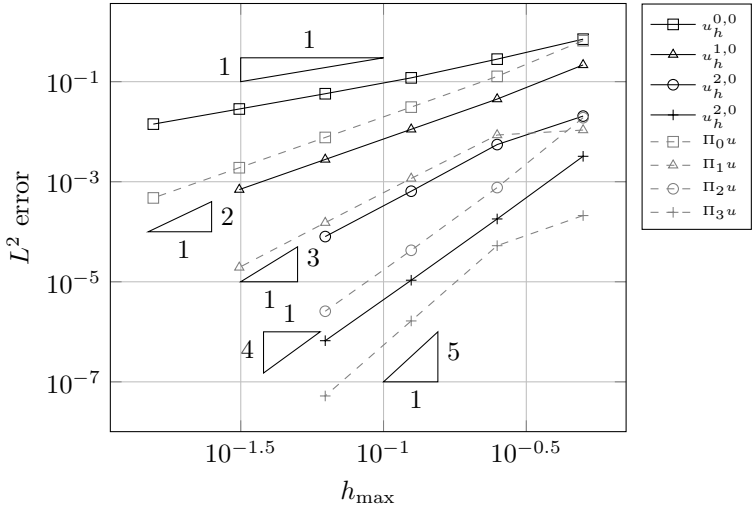

Figure 4: Robustness of the approximation in augmented trial spaces

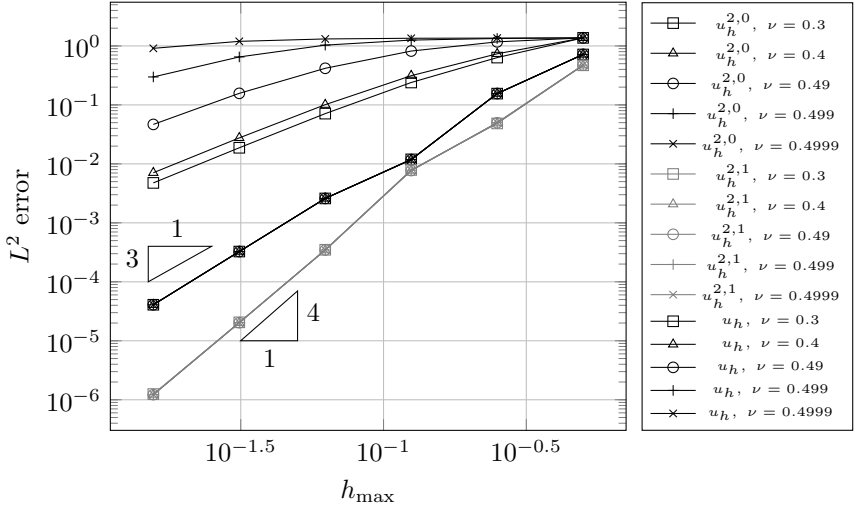

Figure 5: Robustness of the post-processing

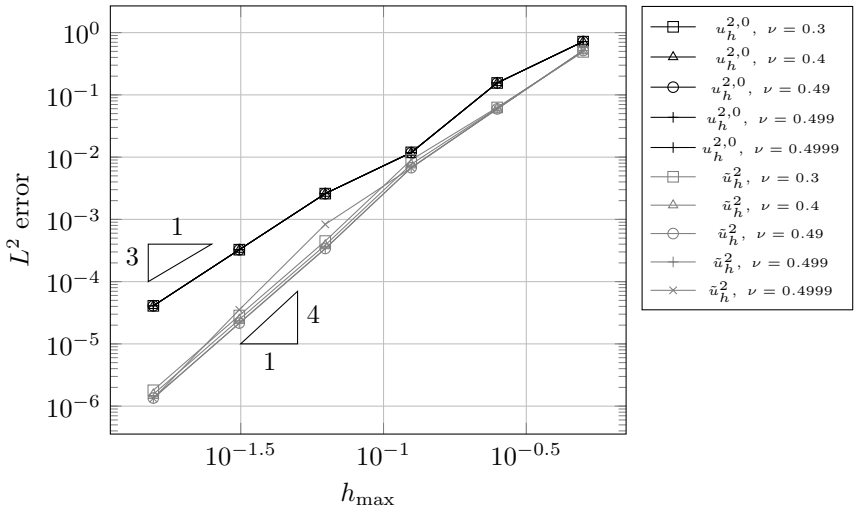

Figure 6: Robustness uniform L-shape for the postprocessing

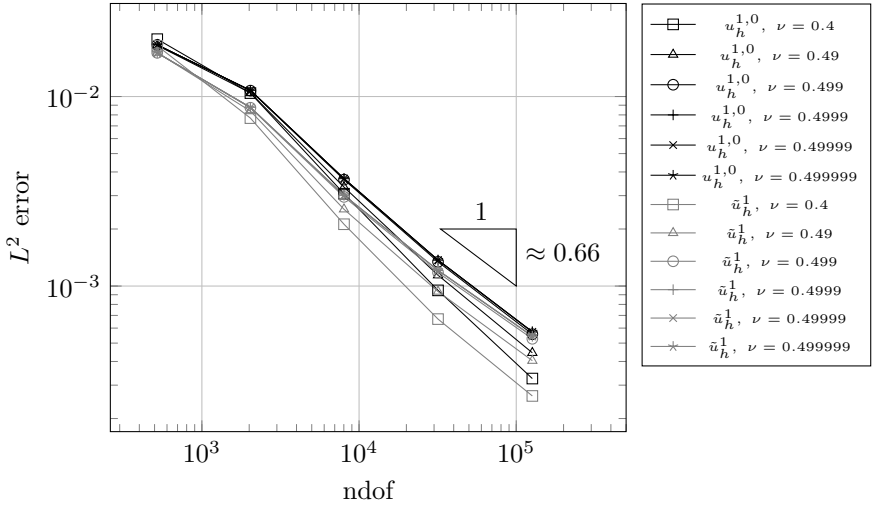

Figure 7: Robustness uniform L-shape for the augmented trial spaces

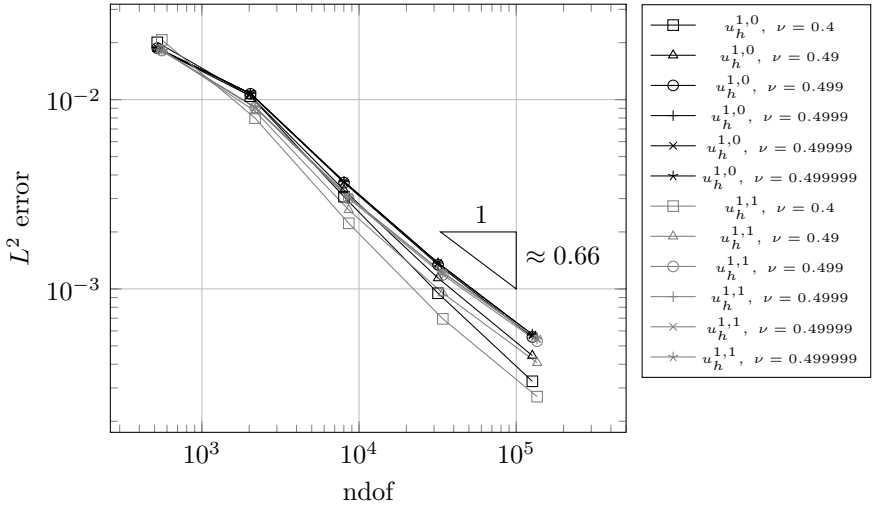

Figure 8: Robustness adaptive L-shape for the augmented trial spaces

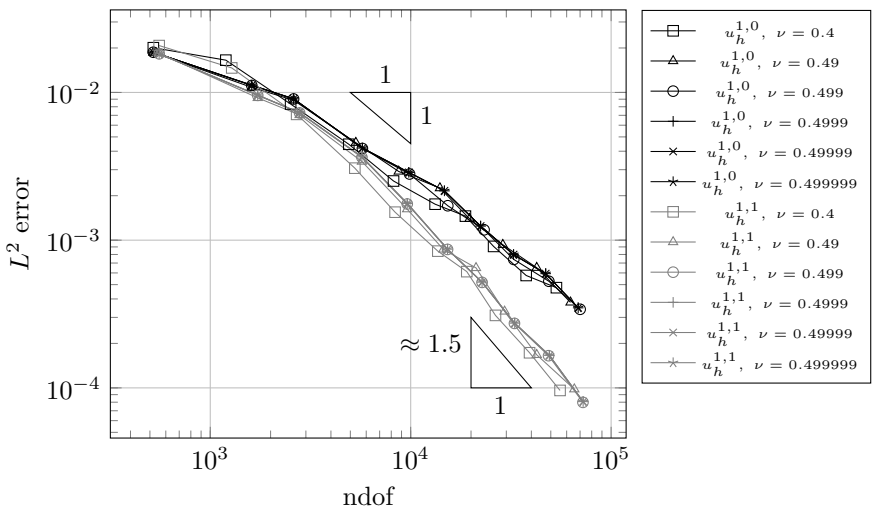

Figure 9: Robustness adaptive L-shape for the postprocessing

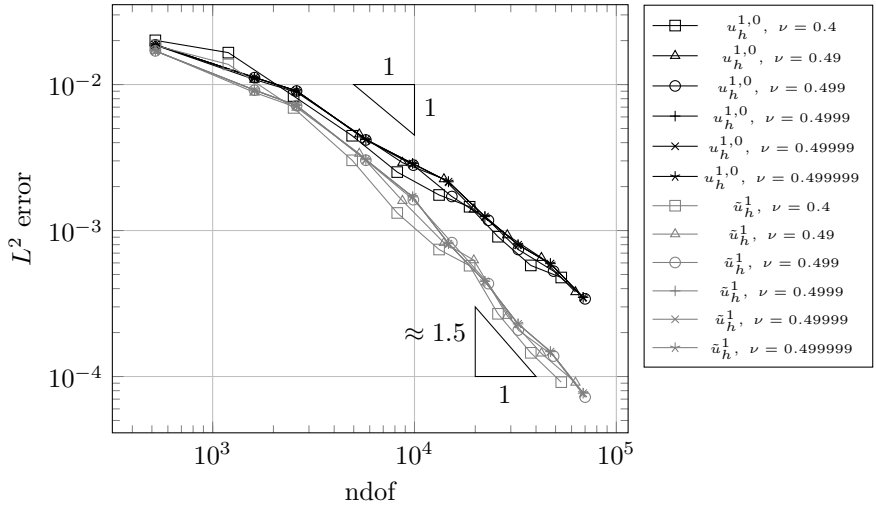



\begin{figure}[h!]
    \centering
    \caption{L-shape adaptive shape after 10 Steps}
    \includegraphics[scale=0.2]{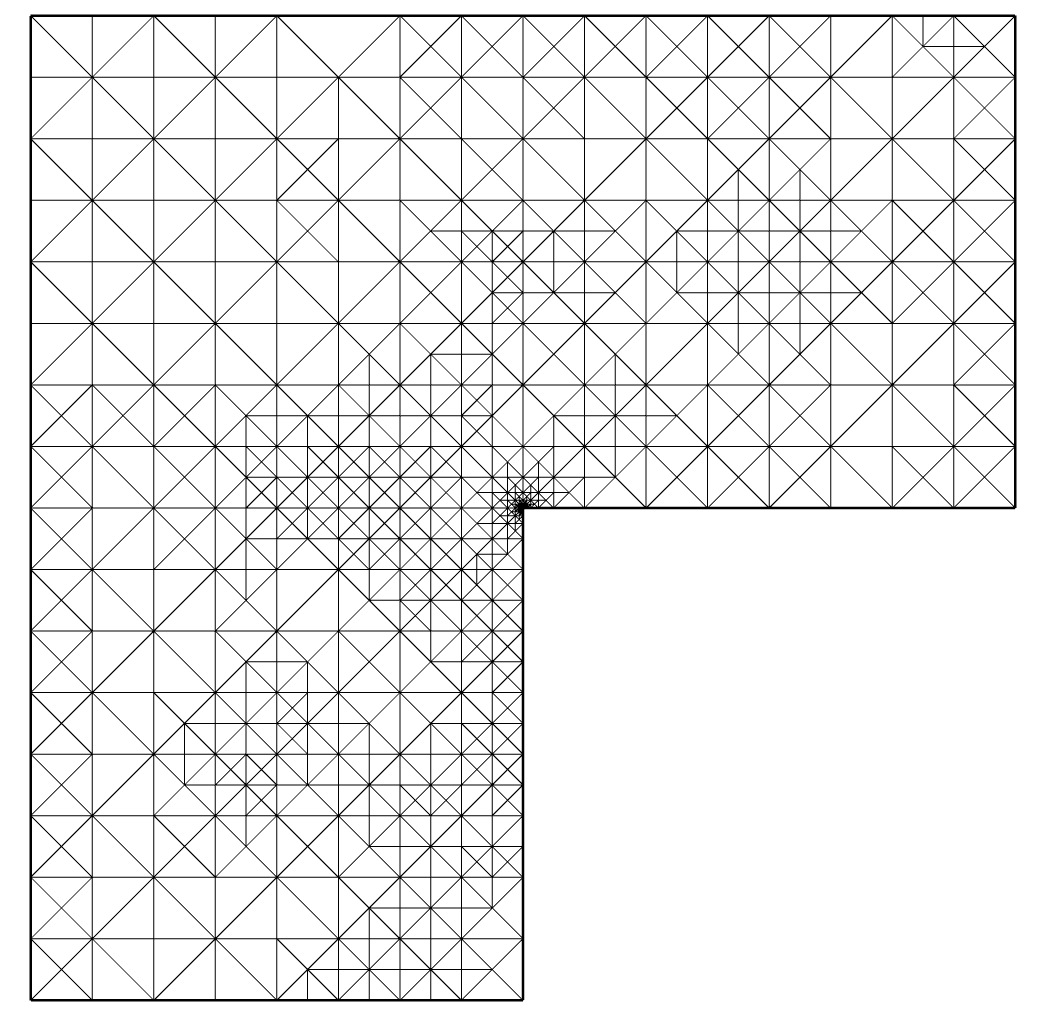}
    \label{fig:adap_mesh}
\end{figure}

\bibliographystyle{abbrv} 
\bibliography{biblo.bib}
\end{document}